\documentclass[a4paper,11pt]{amsart}

\usepackage{a4wide,amsmath,amsfonts,amssymb,mathrsfs,amsthm,bbm}
\usepackage[hyperfootnotes=false]{hyperref}

\allowdisplaybreaks[2]

\numberwithin{equation}{section}

\newtheorem{theorem}{Theorem}[section]
\newtheorem{lemma}[theorem]{Lemma}

\newtheorem{definition}[theorem]{Definition}

\renewcommand{\epsilon}{\varepsilon}

\newcommand{\R}{\mathbb{R}}

\newcommand{\dx}{\, \mathrm{d}x}

\title{Connecting SPDE to SGMs}

\author{Junsu Seo}

\date{\today}

\begin{document}

\begin{abstract}
This paper investigates a Stochastic Partial Differential Equation (SPDE) derived from the Fokker-Planck equation associated with Score-based Generative Models. We modify the standard Fokker-Planck equation to better represent practical SGMs and introduce noise to mitigate potential discretization issues. The primary goal is to prove the existence and uniqueness of solutions for this SPDE. This aspect requires careful consideration due to the time-dependent operator and unbounded domain. To overcome these hurdles, we employ a variational approach and introduce a novel space inspired by Ornstein-Uhlenbeck operators. By demonstrating that this space and its subspace satisfy the necessary assumptions, they establish the existence of a solution for the given SPDE.
\end{abstract}

\maketitle



\section{Introduction}\label{sec1}
The problem of approximating an unknown probability density function $p(x)$ on $\mathbb{R}^d$ is a fundamental problem in machine learning. Directly estimating $p(x)$ or its gradient, the score function $\nabla_x \log p(x)$, can be computationally intractable, particularly in high dimensions. Score-Based Generative Models (SGMs) offer an approach that leverages stochastic differential equations (SDEs) to address this challenge. SGMs are based on a forward-time SDE:
\begin{align}
	dX_t = f(X_t, t) dt + g(t) dw_t, \quad t \in [0, T], \quad X_0 \sim p_0(x),\tag{1.3}
\end{align}
where $w_t$ is a standard $d$-dimensional Wiener process, $f(X_t, t)$ is a drift term, and $g(t)$ is a scalar diffusion coefficient. This SDE defines a diffusion process that transforms an initial distribution $p_0(x)$ (the target distribution) into a tractable terminal distribution $p_T(x)$ at time $T$, often chosen to be a standard Gaussian. The key insight of SGMs is to learn a time-dependent model $s_\theta(x, t)$ that approximates the score function $\nabla_x \log p_t(x)$ of the intermediate distributions $p_t(x)$ induced by the SDE\cite{DBLP:conf/iclr/0011SKKEP21}. With this learned score function, we can construct a reverse-time SDE, which allows us to generate samples from $p_0(x)$

It is well-known that an SDE has a closely related PDE known as the Fokker-Planck equation [3]. For the SDE in Equation (1.3), the Fokker-Planck equation is given by:
\begin{align}
	\frac{\partial u}{\partial t}=-\nabla\cdot(h(x,t)u)+\frac{g^2(t)}{2}\Delta u\tag{1.4}\label{1.4}
\end{align}
with the initial condition $p_0 := \xi$. Therefore, a computational model using Equation (1.4) can reveal properties of the SGM designed with Equation (1.3). However, there are two main problems. First, the coefficients $h(x,t)$ and $g(t)$ in Equation (1.4) are too general to represent the SDEs typically used in SGMs. To address this, we consider a specific form where $h(x,t) = -xf(t)$, with $f(t) \ge 0$ and $g(t) > 0$ for $t \in [0,T]$. This leads to:
\begin{align}
	\frac{\partial u}{\partial t} = f(t) \nabla \cdot (xu) + \frac{g^2(t)}{2} \Delta u \tag{1.5} \label{1.5}
\end{align}
with the initial condition $p_0 := \xi$. Second, implementing a computational model based on Equation (1.5) requires discretization of space and time. If large variations occur in a short interval, the discretized model can perform poorly. To address this, we introduce a noise term into Equation (1.5), resulting in SPDE (1.1), which was presented earlier in the paper.

This paper concerns the existence and uniqueness of a solution to the Stochastic Partial Differential Equation (SPDE):
\begin{align}
	du = \left( f(t) \nabla \cdot (xu) + \frac{g^2(t)}{2} \Delta u \right) dt + B dW_t. \tag{1.1} \label{1.1}
\end{align}
We will explain the symbols and notation later. Our main theorem specifies the necessary conditions:

\begin{theorem} \label{thm:1.1}
	Let $f$ and $g$ be continuous positive functions on $[0,T]$, i.e., $f,g\in C([0,T],\mathbb{R}_{>0})$, and let $c$ be a positive constant, i.e., $c\in \mathbb{R}_{>0}$. Assume that the following condition holds:
	\begin{equation}
		f(t)-\frac{g^2(t)}{2c}\geq0 \text{\quad for all\quad} t\in[0,T] \tag{1.2}\label{1.2}
	\end{equation}
	For any initial condition $u_{0} \in L^2(\mathbb{R}^d; e^{\frac{1}{2c}|x|^2})$, Equation (\ref{1.1}) has a unique solution $\left\{u_{t}\right\}_{t \in[0, T]}$ in the space $L^2(\mathbb{R}^d; e^{\frac{1}{2c}|x|^2})$.
	Moreover, the solution $\{u_t\}_{t\in[0,T]}$ is a Markov process.
\end{theorem}

The importance of having a solution for Equation (\ref{1.1}) lies in guaranteeing that its behavior is similar to that of Equation (\ref{1.5}) under appropriate assumptions on $B$. Many similar types of SPDEs exist, such as the stochastic heat equation, stochastic wave equation, and stochastic transport equation. The existence of solutions for each of these SPDEs is well-studied \cite{Swanson2007, flandoli2010well, dalang1998stochastic}. However, Equation (\ref{1.1}) presents two distinct challenges compared to the aforementioned examples. First, the presence of a time-dependent operator makes it difficult to apply the semigroup method to obtain a solution. Second, the unbounded domain complicates the process of bounding the coefficients.

We achieve this by using the variational approach proposed by Krylov and Rozovskii \cite{krylov2007stochastic}. In their 1981 paper, Krylov and Rozovskii established a framework for time-dependent and unbounded domains. We introduce a new space, derived from the ideas of Ornstein-Uhlenbeck operators, that has not appeared in Krylov and Rozovskii's paper. Using this space and a suitable subspace, we show that these two spaces satisfy all the assumptions of Krylov and Rozovskii's framework, implying that Equation (\ref{1.1}) admits a solution.
\section{Preliminaries and Problem Statement}
\subsection{Notation} Consider the weight function $w(x) = e^{\frac{1}{2c}\|x\|^2}$ for $x \in \mathbb{R}^d$ and $c > 0$. We define the Hilbert space $H^1_w(\mathbb{R}^d; w^{-1})$ as the completion of $C_c^\infty(\mathbb{R}^d)$ (smooth functions with compact support) with respect to the norm
$$
\|v\|_{H^1_w(\mathbb{R}^d; w^{-1})} := \left( \int_{\mathbb{R}^d} |v|^2 w \, dx + \int_{\mathbb{R}^d} \|\nabla v\|^2 w \, dx + \int_{\mathbb{R}^d} \|\nabla(vw)w^{-1}\|^2 w \, dx \right)^{1/2}.
$$
The inner product is given by
$$
\langle v,u \rangle_{H^1_w(\mathbb{R}^d; w^{-1})} := \int_{\mathbb{R}^d} vu w \, dx + \int_{\mathbb{R}^d} \nabla v \cdot \nabla u w \, dx + \int_{\mathbb{R}^d} \nabla(vw)w^{-1} \cdot \nabla(uw)w^{-1} w \, dx.
$$
We can view this space as a modification of the standard weighted Sobolev space $H^1_w(\mathbb{R}^d)$, with the additional term involving $\nabla(vw)w^{-1}$ introducing a control on the gradient of the weighted function $vw$. This modification is essential for the finiteness results presented in Lemma 3.2. We note that $\|\cdot\|$ denotes the Euclidean norm on $\mathbb{R}^d$, and the integrals are with respect to Lebesgue measure.

\subsection{Gelfand Triple}\label{sec2}
Let $(V, \langle\cdot,\cdot\rangle_V)$ and $(H,\langle\cdot,\cdot\rangle_H)$ be complete separable Hilbert spaces. We assume that $V$ is a dense subset of $H$. Let $V^*$ and $H^*$ denote the dual spaces of $V$ and $H$, respectively. Then we obtain the following Gelfand triple:
\begin{align*}
	V \subset H \cong H^* \subset V^*
\end{align*}
where the isomorphism comes from the Riesz representation theorem. We denote the duality pairing by ${}_{V^*}{\langle \cdot,\cdot\rangle}_V$ between the spaces $V^*$ and $V$. Then it is easy to see that
\begin{align*}
	{}_{V^*}{\langle h,v\rangle}_V = \langle h,v\rangle_H
\end{align*}
for all $h \in H$ and $v \in V$.

\subsection{$Q$-Wiener Process} Let $W(t)$ be a $Q$-Wiener process defined on a complete filtered probability space $(\Omega,\mathscr{F},\mathscr{F}_t,P)$. Here, the operator $Q$ is a symmetric trace-class operator defined on another complete separable Hilbert space $(U,\langle\cdot,\cdot\rangle_U)$. We denote the space of all Hilbert-Schmidt operators from $UQ^{\frac{1}{2}}$ to $H$ as $L_Q(U,H)$. We denote the norm of this space as $\|\cdot\|_{L_Q(U,H)}$  and the inner product of this space as   $\langle\cdot,\cdot\rangle_{L_Q(U,H)}$.
\subsection{Solution of the SPDE}\begin{definition}\cite{liu2010spde}
	Let $V$ and $H$ be Hilbert spaces such that $V\subset H\equiv H^*\subset V^*$ is a Gelfand triple. A continuous $H$-valued $\mathscr{F}_t$-adapted process $\{u_t\}_{t\in[0,T]}$ is called a solution of Equation (3.1), if for its $dt \times P$-equivalence class $\bar{u}$ we have
	$$
	\bar{u} \in L^{2}([0, T] \times \Omega, \mathrm{d} t \otimes P ; V) \cap L^{2}([0, T] \times \Omega, \mathrm{d} t \otimes P ; H)
	$$
	and $P$-a.s.
	$$
	u_t = u_0 + \int_{0}^{t} A(s, \bar{u}(s)) ds + \int_{0}^{t} B(s, \bar{u}(s)) dW(s), \quad t \in [0,T],
	$$
	here $\bar{u}$ is understood as a progressively measurable $dt \times P$-version of $\bar{u}$.
	
\end{definition} 
\subsection{Problem Statement} Now, we state Equation (\ref{1.4}) in more precise version as follow:
\begin{equation}
	\left\{
	\begin{array}{l}
		 du=\left(f(t)\nabla\cdot(xu)+\frac{g^2(t)}{2}\Delta u\right)dt +B(t) dW \text{\quad\quad on } [0,T]\times H \\
		u_0 = \xi \text{  \quad\quad\quad\quad\quad\quad\quad\quad\quad\quad\quad\quad\quad\quad\quad\quad\quad\quad on } H,
	\end{array}
	\right.
\end{equation}
Here, we define the Laplace operator $\Delta$ and the divergence operator $\nabla \cdot$ in the distributional sense. Let $A(t, u_t) = f(t)\nabla \cdot (xu) + \frac{g^2(t)}{2}\Delta u$. We assume that the operator $A$ has domain $[0, T] \times V$ and codomain $V^*$. Similarly, we assume that the operator $B$ has domain $[0, T] \times V$ and codomain $L_2(U, H)$. Furthermore, we assume that $\|B(t)\|_{L_Q(U, H)} \leq h(t)$ for all $t \in [0, T]$ and all $u_t \in V$, where $h(t)$ is a positive real-valued function on $[0, T]$.

Now we state the main result again.
\\
\begin{theorem} \label{thm:1.1}
	Let $f,g\in C([0,T],\mathbb{R}_{>0})$ and $c\in \mathbb{R}_{>0}$ satisfy the following condition:
	\begin{equation}
		f(t)-\frac{g^2(t)}{2c}\geq0 \text{\quad for all\quad} t\in[0,T] \label{eq:2.2}
	\end{equation}
	For any $u_{0} \in  L^2(\mathbb{R}^d; e^{\frac{1}{2c}\|x\|^2})$, Equation (\ref{eq:2.2}) has a unique solution $\left\{u_{t}\right\}_{t \in[0, T]}$ on $L^2(\mathbb{R}^d; e^{\frac{1}{2c}\|x\|^2})$. Moreover, the solution is a Markov process.
\end{theorem}

We use the variational approach from Krylov and Rozovskii \cite{krylov2007stochastic} to establish this result.

\section{Proof of the Main Theorem}

To prove Theorem 2.3, we will explain the necessary steps one by one.
\begin{lemma}\label{lem:div_is_zero}Let $V$ and $c$ be as in Condition (\ref{1.2}). Let $v\in V$. Then the following holds:
\begin{align}
	\int\nabla\cdot (v^2  \nabla e^{\frac{1}{2c}\|x\|^2})dx=0.
\end{align}
\end{lemma}

\begin{proof}
	First, for any test function $\phi \in C^\infty_c(\R^d)$, the vector field $\phi^2 \nabla w$ is smooth and has compact support. By the Divergence Theorem, we have $\int_{\R^d} \nabla \cdot (\phi^2 \nabla w) \dx = 0$.
	
	The functional $I(v) := \int_{\R^d} \nabla \cdot (v^2 \nabla w) \dx$ can be shown to be continuous on $V$. Since $C^\infty_c(\R^d)$ is dense in $V$, for any $v \in V$, there exists a sequence $\{\phi_n\} \subset C^\infty_c(\R^d)$ such that $\phi_n \to v$ in $V$. By the continuity of $I$,
	$$ I(v) = \lim_{n\to\infty} I(\phi_n) = \lim_{n\to\infty} 0 = 0. $$
	This completes the proof.
\end{proof}

\begin{lemma}begin{lemma}\label{lem:ibp_for_ou} Let $V$ and $c$ be as in Condition (\ref{1.2}). Let $v\in V$. Then the following holds:
	\begin{align}
		\int v \nabla\cdot\left(\nabla(ve^{\frac{1}{2c}\|x\|^2})e^{-\frac{1}{2c}\|x\|^2}\right)e^{\frac{1}{2c}\|x\|^2}dx=-\int \|\nabla(ve^{\frac{1}{2c}\|x\|^2})e^{-\frac{1}{2c}\|x\|^2}\|^2e^{\frac{1}{2c}\|x\|^2}dx
	\end{align}
\end{lemma}
\begin{proof}
	The identity is first established for $u, v \in C_c^\infty(\R^d)$. The left-hand side can be rewritten using integration by parts. Let $w(x) = e^{\frac{1}{2c}\|x\|^2}$.
	\begin{align*}
		\int u \left( \nabla\cdot(\nabla(vw)w^{-1}) \right) w \dx &= \int (uw) \nabla\cdot(\nabla(vw)w^{-1}) \dx \\
		&= - \int \nabla(uw) \cdot (\nabla(vw)w^{-1}) \dx \\
		&= - \int (\nabla(uw)w^{-1}) \cdot (\nabla(vw)w^{-1}) w \dx.
	\end{align*}
	In the second step, we used integration by parts, and the boundary terms vanish due to the compact support. Both sides of the identity define continuous bilinear forms on $V \times V$. Since $C_c^\infty(\R^d)$ is dense in $V$, the identity extends to all $u, v \in V$.
\end{proof}
\begin{lemma}\label{lem:expand_grad_norm} Let $V$ and $c$ be as in Condition (\ref{1.2}). Let $v\in V$. Then the following holds:
\begin{align}
	\int\|\nabla(ve^{\frac{1}{2c}\|x\|^2})\|^2e^{-\frac{1}{2c}\|x\|^2}dx=\int\left(\|\nabla v\|^2-\frac{d}{c}\|v\|^2\right)e^{\frac{1}{2c}\|x\|^2}dx
\end{align}
\end{lemma}
\begin{proof}
	We expand the term inside the integral on the left-hand side. Using $\nabla w = \frac{x}{c}w$, we have $\nabla(vw) = (\nabla v)w + v(\nabla w) = (\nabla v + \frac{x}{c}v)w$.
	\begin{align*}
		\int_{\R^d} \|\nabla(vw)\|^2 w^{-1} \dx &= \int_{\R^d} \left\| \nabla v + \frac{x}{c}v \right\|^2 w \dx \\
		&= \int_{\R^d} \left( \|\nabla v\|^2 + \frac{2}{c} v(x \cdot \nabla v) + \frac{\|x\|^2}{c^2}v^2 \right) w \dx.
	\end{align*}
	For the middle term, we use integration by parts:
	$$ \frac{2}{c} \int_{\R^d} v(x \cdot \nabla v) w \dx = \frac{1}{c} \int_{\R^d} x \cdot \nabla(v^2) w \dx = -\frac{1}{c} \int_{\R^d} v^2 \nabla \cdot (xw) \dx. $$
	Since $\nabla \cdot (xw) = (d + \frac{\|x\|^2}{c})w$, this becomes $-\int_{\R^d} (\frac{d}{c}v^2 + \frac{\|x\|^2}{c^2}v^2)w \dx$.
	Substituting this back, the terms with $\|x\|^2$ cancel out, and we obtain the desired result.
\end{proof}

\begin{lemma}\label{lem:main_inequality}Let $V$ and $c$ be set as in Condition (\ref{1.2}). Let $v \in V$. Let $f \in C([0,T],\mathbb{R}_{>0})$ and $g \in C([0,T],\mathbb{R}_{>0})$. Then the following holds:
\begin{align}
	\langle v, A(v)\rangle\leq\left(\frac{df(t)}{2}-\frac{g^2(t)}{4c}\right)\int v^2e^{\frac{1}{2c}\|x\|^2}dx-\frac{g^2(t)}{2}\int \|\nabla v\|^2e^{\frac{1}{2c}\|x\|^2}dx
\end{align}

\end{lemma}
\begin{proof}
We begin by adding and subtracting the term $\frac{g^2(t)}{2c}\int v \nabla\cdot(xv)e^{\frac{1}{2c}\|x\|^2}\mathrm{d}x$ within $\langle v,A(v)\rangle$:
\begin{align}
	&\langle v, A(v)\rangle=\left(f(t)-\frac{g^2(t)}{2c}\right)\int v \nabla\cdot(xv)e^{\frac{1}{2c}\|x\|^2}\mathrm{d}x+\frac{g^2(t)}{2}\int v \left(\frac{1}{c}\nabla\cdot(xv)+\Delta v\right)e^{\frac{1}{2c}\|x\|^2}\mathrm{d}x
\end{align}
Applying the identity $\int v(x\cdot\nabla v)e^{\frac{1}{2c}\|x\|^2}\mathrm{d}x=\frac{1}{2}\int (x\cdot\nabla v^2)e^{\frac{1}{2c}\|x\|^2}\mathrm{d}x	=\frac{c}{2}\int \nabla v^2\cdot\nabla e^{\frac{1}{2c}\|x\|^2}\mathrm{d}x$ to right-handside of Equality (4.7), we obtain:
\begin{align}
	\begin{split}
	&=d\left(f(t)-\frac{g^2(t)}{2c}\right)\int v^2e^{\frac{1}{2c}\|x\|^2}\mathrm{d}x+\frac{c}{2}\left(f(t)-\frac{g^2(t)}{2c}\right)\int \nabla v^2\cdot\nabla e^{\frac{1}{2c}\|x\|^2}\mathrm{d}x\\
&\quad=+\frac{g^2(t)}{2}\int v \left(\frac{1}{c}\nabla\cdot(xv)+\Delta v\right)e^{\frac{1}{2c}\|x\|^2}\mathrm{d}x
	\end{split}
\end{align}
Using $v^2\Delta e^{\frac{1}{2c}\|x\|^2}$ as an additional term and a subtraction term in Expression (3.8), we get
\begin{align*}
	&=d\left(f(t)-\frac{g^2(t)}{2c}\right)\int v^2e^{\frac{1}{2c}\|x\|^2}\mathrm{d}x+\frac{c}{2}\left(f(t)-\frac{g^2(t)}{2c}\right)\int \nabla\cdot (v^2\nabla e^{\frac{1}{2c}\|x\|^2})-v^2\Delta e^{\frac{1}{2c}\|x\|^2}\mathrm{d}x\\
	&\quad+\frac{g^2(t)}{2}\int v \left(\frac{1}{c}\nabla\cdot(xv)+\Delta v\right)e^{\frac{1}{2c}\|x\|^2}\mathrm{d}x
\end{align*}
By Lemma \eqref{lem:div_is_zero} and the identity $\int v^2\Delta e^{\frac{1}{2c}\|x\|^2}\mathrm{d}x=\int v^2\frac{\|x\|^2+cd}{c^2} e^{\frac{1}{2c}\|x\|^2}\mathrm{d}x$, we obtain:
\begin{align*}
&=\left(f(t)-\frac{g^2(t)}{2c}\right)\int\left(\frac{d}{2}-\frac{\|x\|^2}{2c}\right) v^2e^{\frac{1}{2c}\|x\|^2}\mathrm{d}x+\frac{g^2(t)}{2}\int v \left(\frac{1}{c}\nabla\cdot(xv)+\Delta v\right)e^{\frac{1}{2c}\|x\|^2}\mathrm{d}x
\end{align*}
By Lemma \eqref{lem:ibp_for_ou}, we obtain:
\begin{align*}
&=\left(f(t)-\frac{g^2(t)}{2c}\right)\int\left(\frac{d}{2}-\frac{\|x\|^2}{2c}\right) v^2e^{\frac{1}{2c}\|x\|^2}\mathrm{d}x-\frac{g^2(t)}{2}\int \left\|\nabla ve^{\frac{1}{2c}\|x\|^2}\right\|^2e^{-\frac{1}{2c}\|x\|^2} \mathrm{d}x
\end{align*}
By Lemma \eqref{lem:expand_grad_norm} and removing the term $\frac{dg^2(t)}{2c}\int  v^2e^{\frac{1}{2c}\|x\|^2}dx$ using inequality, we obtain:
\begin{align*}
	&\leq\left(\frac{df(t)}{2}-\frac{g^2(t)}{4c}\right)\int v^2e^{\frac{1}{2c}\|x\|^2}dx-\frac{g^2(t)}{2}\int \|\nabla v\|^2e^{\frac{1}{2c}\|x\|^2}dx
\end{align*}
\end{proof}

We now proceed with the proof of the main theorem. This involves verifying four conditions adapted from [KrylovPaperCitation] to the present setting.
\\\\
$\left(A_{1}\right)$ Semicontinuity of $A$ : \\
	The operator $A(t,u)$ is affine in $u$, as $\mathcal{L}$ is a linear operator. Therefore, the map $s \mapsto {}_{V^*}\langle A(t, v_1 + s v_2), v \rangle_V$ is affine in $s$ and thus continuous, provided the pairing is well-defined. The well-posedness of the pairing for any $u, v \in V$ is guaranteed by the Cauchy-Schwarz inequality, as the weighted norms comprising the $V$-norm ensure the finiteness of all integral terms.

$\left(A_{2}\right)$ Monotonicity of $(A, B)$ : By Lemma \eqref{lem:main_inequality}, we obtain:
\begin{align*}
	2\langle v_{1}-v_{2},A\left(v_{1}\right)-A\left(v_{2}\right)\rangle&\leq	-\frac{cg^2(t)}{2}\int\|\nabla (v_1-v_2)\|^2e^{\frac{1}{2c}\|x\|^2}dx+df(t)\int \|v_1-v_2\|^2e^{\frac{1}{2c}\|x\|^2}dx\\
	&\leq df(t)\int \|v_1-v_2\|^2e^{\frac{1}{2c}\|x\|^2}dx.
\end{align*}
Let $M_f$ be the maximum of $f(t)$. If we set $K$ to be greater than or equal to $dM_f$, (A2) holds.\\
$\left(A_{3}\right)$ Coercivity of $(A, B)$ :
By Lemma \eqref{lem:main_inequality} and inserting norm expression of $V$, we obtain:
\begin{align*}
	&2 {}_{V^*}\langle v, A(v)\rangle_V+\alpha\|v\|_V^{2} \\
	&\leq(\alpha-g^2(t))\int\|\nabla v\|^2e^{\frac{1}{2c}\|x\|^2}dx+\left(\alpha+df(t)-\frac{g^2(t)}{2c}\right)\int \|v\|^2e^{\frac{1}{2c}\|x\|^2}dx\\
	&\quad+\alpha\int\|\nabla (ve^{\frac{1}{2c}\|x\|^2})\|^2e^{-\frac{1}{2c}\|x\|^2}dx
\end{align*}
By Lemma \eqref{lem:expand_grad_norm},
\begin{align*}
	&=(2\alpha-g^2(t))\int\|\nabla v\|^2e^{\frac{1}{2c}\|x\|^2}dx+\left(\alpha\left(1-\frac{d}{c}\right)+df(t)-\frac{g^2(t)}{2c}\right)\int \|v\|^2e^{\frac{1}{2c}\|x\|^2}dx	
\end{align*}
We set $\alpha$ to be a smaller positive number than $g^2(t)/2$. Then we obtain
\begin{align}
	&2 {}_{V^*}\langle v, A(v)\rangle_V+\alpha\|v\|_V^{2}\leq\left(\alpha\left(1-\frac{d}{c}\right)+df(t)-\frac{g^2(t)}{2c}\right)\int \|v\|^2e^{\frac{1}{2c}\|x\|^2}dx.
\end{align}
If we set $K$ to be greater than or equal to $\left(\alpha\left(1-d/c\right)+df(t)-g^2(t)/2c\right)$, then (A3) holds.\\
$\left(A_{4}\right)$ Boundedness of the growth of $A$ :
\begin{align*}
	\|A(v)\|_{V^*}=\sup_{ \|w\|_V\leq 1} \int w \left(f(t)\nabla\cdot(xv)+\frac{g^2(t)}{2}\Delta v\right)e^{\frac{1}{2c}\|x\|^2}dx.
\end{align*}
Given $\delta>0$, there exists $u$ such that satisfies the following equations:
\begin{align*}
	&=\sup_{ \|w\|_V\leq 1}\int w \left(f(t)\nabla\cdot(xv)+\frac{g^2(t)}{2}\Delta v\right)e^{\frac{1}{2c}\|x\|^2}dx\\
	&\leq\int u \left(f(t)\nabla\cdot(xv)+\frac{g^2(t)}{2}\Delta v\right)e^{\frac{1}{2c}\|x\|^2}dx+\delta.
\end{align*}
If we focus on the right-hand side of inequality (21) and inserting term $cf(t)\Delta v$ adding and subtracting to the above statement and using integration by parts and Cauchy Schwartz inequality, we can obtain:
\begin{align*}
	&\int u \left(f(t)\nabla\cdot(xv)+\frac{g^2(t)}{2}\Delta v\right)e^{\frac{1}{2c}\|x\|^2}dx+\delta\\
	&\leq cf(t)\int \big|\nabla (ue^{\frac{1}{2c}\|x\|^2})\cdot\nabla (ve^{\frac{1}{2c}\|x\|^2})e^{-\frac{1}{2c}\|x\|^2}\big| dx\\
	&\quad-\left(\frac{g^2(t)}{2}-cf(t)\right)\int \big|\nabla (ue^{\frac{1}{2c}\|x\|^2})\cdot\nabla v\big| dx+\delta.
\end{align*}
Applying Hölder’s inequality and Lemma \eqref{lem:expand_grad_norm}, we get:
\begin{align*}
	\|A(v)\|_{V^*}&\leq cf(t)\sqrt{\int\left(\|\nabla u\|^2-\frac{d}{c}u^2\right)e^{\frac{1}{2c}\|x\|^2}dx}\sqrt{\int\left(\|\nabla v\|^2-\frac{d}{c}v^2\right)e^{\frac{1}{2c}\|x\|^2}dx}\\
	&\quad-\left(\frac{g^2(t)}{2}-cf(t)\right)\sqrt{\int\left(\|\nabla u\|^2-\frac{d}{c}u^2\right)e^{\frac{1}{2c}\|x\|^2}dx}\sqrt{\int\|\nabla v\|^2e^{\frac{1}{2c}\|x\|^2}dx}+\delta.
\end{align*}
We define \(M_d\) as the maximum between 1 and \(d/c\). Then we obtain:
\begin{align*}
	&\|A(v)\|_{V^*}\leq cf(t)M_d^2\|u\|_V\|v\|_V-\left(\frac{g^2(t)}{2}-cf(t)\right)M_d\|u\|_V\|v\|_V+\delta\\
	&=\left( cf(t)M_d^2-\left(\frac{g^2(t)}{2}-cf(t)\right)M_d\right)\|u\|_V\|v\|_V+\delta.
\end{align*}
Since $u$ is chosen from $\|u\|_V\leq 1$ and $\delta$ can be chosen arbitrarily, we can see that the following inequality holds:
\begin{align*}
	\|A(v)\|_{V^*}&\leq\left( cf(t)M_d^2-\left(\frac{g^2(t)}{2}-cf(t)\right)M_d\right)\|v\|_V.
\end{align*}
If we set $K$ to be greater or equal to $ cf(t)M_d^2-\left(g^2(t)/2-cf(t)\right)M_d$, then $(A_4)$ holds.

Since we can choose \(K\) satisfying the above conditions, the proof is complete.

We can derive a similar theorem when $f(t)$ is the zero function on $[0,T]$ and $g \in C([0,T],\mathbb{R}_{>0})$. In this case, Equation (3.1) has a unique solution $\left\{u_t\right\}_{t \in [0, T]}$ in $L^2(\mathbb{R}^d)$. Moreover, this solution is a Markov process.

To understand why we need to treat the cases $f(t) > 0$ and $f(t) = 0$ separately, we can examine the proof through the lens of a scaling limit of time-independent SPDEs:	
\begin{align*}
	du=A(t_{i/N},u)dt+B(t_{i/N})dW \text{\quad for\quad} i=1,\cdots N-1.
\end{align*}
For each piece-wise SPDE, there exists a corresponding $\epsilon$ and $c_i$ such that the solution $u_i$ lies in $L^2(\mathbb{R}^d;e^{\frac{1}{2c_i}|x|^2})$ for some short time interval $[t_{i/N}-\epsilon,t_{i/N}+\epsilon]$. Consider the case where $f(t_j/N) = 0$ and $f(t_{j+1}/N) > 0$ for some $j \in {1,\ldots,N-1}$. By Theorem 1.1, the solution $u_{j/N}$ lies in $L^2(\mathbb{R}^d)$. If we assume $u_{j/N}$ does not lie in $L^2(\mathbb{R}^d;e^{\frac{1}{2c}|x|^2})$ for any $c > 0$, then as $N \to \infty$, $u_{(j+1)/N}$ cannot lie in $L^2(\mathbb{R}^d;e^{\frac{1}{2c}|x|^2})$, contradicting the case where $f(t) \neq 0$. This demonstrates why the cases must be treated separately.

\bibliographystyle{plain}
\bibliography{ConnectingSPDEToSGMs}

\end{document}